%% file: Molina_Jaume_NDT_2017_arxiv.tex
\documentclass[a4paper,11pt]{article}

\usepackage{amscd,amsmath,amsthm,amsfonts}
\usepackage{graphicx}
\usepackage{amssymb,color,amsbsy}%, amsthm}
\usepackage[matrix,arrow]{xy}
\usepackage{mathabx}
\usepackage{natbib}

\usepackage{tikz}
\usepackage{caption}
\usepackage[pdftex]{hyperref}
\hypersetup{colorlinks=false}
\newtheorem{theorem}{Theorem}[section]

\newtheorem{algorithm}[theorem]{Algorithm}

\newtheorem{corollary}[theorem]{Corollary}

\newtheorem{definition}[theorem]{Definition}

\newtheorem{lemma}[theorem]{Lemma}

\newtheorem{proposition}[theorem]{Proposition}

\DeclareMathAlphabet{\mathpzc}{OT1}{pzc}{m}{it}

\newcommand{\Supp}[2]{\operatorname{\mathpzc{Supp}}_{#1}(#2)}
\newcommand{\N}[1]{\operatorname{\mathpzc{N}}(#1)}

\newcommand{\Core}[1]{\operatorname{\mathpzc{Core}}(#1)}

\newcommand{\ConnE}[1]{\operatorname{\mathpzc{Conn-E}}(#1)}

\newcommand{\supp}[2]{\operatorname{supp}_{#1}(#2)}
\newcommand{\nulidad}[1]{\operatorname{null}(#1)}
\newcommand{\rank}[1]{\operatorname{rk}(#1)}
\newcommand{\core}[1]{\operatorname{core}(#1)}

\newcommand{\sgn}[1]{\operatorname{sgn}(#1)}

\newcommand{\down}[3]{{#1}\!\downharpoonleft_{\scalebox{0.5}{#3}}^{\scalebox{0.5}{#2}}}
\newcommand{\up}[3]{{#1}\!\upharpoonleft_{\scalebox{0.5}{#3}}^{\scalebox{0.5}{#2}}}

\begin{document}

	% Titulo:
	\title{Null Decomposition of Trees}
	\date{May 8, 2017}
	\author{Daniel A. Jaume\footnote{Corresponding author, e-mail: djaume@unsl.edu.ar}, Gonzalo Molina \footnote{Universidad Nacional de San Luis. Argentina}}
	%\author{Gonzalo Molina}
	\maketitle

\begin{abstract}
	Let \(T\) be a tree, we show that the null space of the adjacency matrix of \(T\) has relevant information about the structure of \(T\). We introduce the Null Decomposition of trees, and use it in order to get formulas for independence number and matching number of a tree. We also prove that the number of maximum matchings in a tree is related to the null decomposition.
\end{abstract}
%%%%%%%%%%%%%%%%%%%%%%%%%%%%%%%%%%%%%%%%%%%%%%%%%%%%%%%%
%%%%%%%%%%%%%%%%%%%%%%%%%%%%%%%%%%%%%%%%%%%%%%%%%%%%%%%%
%\begin
%\subjclass[05C05]{ 15A03}

%%\end{}
%\begin{keyword}Trees\sep%
%	eigenvectors\sep
%	null space\sep
%	independence number\sep
%	matching number\sep
%	maximum matching.
%	\MSC 05C05 \sep 65H17
%\end{keyword}
%
%
%
%\begin{frontmatter}
%	\title{Null Decomposition of Trees}
%	
%	 \author[daj]{Daniel A Jaume\corref{cor1}}%\fnref{fn1}}
%	 \ead{djaume@unsl.edu.ar}
%	 
%	 \author[daj]{Gonzalo Molina }%\corref{cor1}}%\fnref{fn2}}
%	 \ead{lgmolina@unsl.edu.ar}
%	 
%	 \cortext[cor1]{Corresponding author: Daniel A. Jaume}
%	 %\cortext[cor2]{Principal corresponding author}
%	 %\fntext[fn1]{This is the specimen author footnote.}
%	 %\fntext[fn2]{Another author footnote, but a little more longer.}
%	 %\fntext[fn3]{Yet another author footnote. Indeed, you can have
%	% 	any number of author footnotes.}
%	 
%	 \address[daj]{	Departamento de Matem\'{a}ticas. 
%	 	Facultad de Ciencias F\'{\i}sico-Matem\'{a}ticas y Naturales. Universidad Nacional de San Luis. 
%	 	1er. piso, Bloque II, Oficina 54. 	 	Ej\'{e}rcito de los Andes 950. 
%	 	San Luis, Rep\'{u}blica Argentina. 
%	 	D5700HHW.}
%	
%	\date{Received: date / Accepted: date}
%	% The correct dates will be entered by the editor
%\end{frontmatter}

%%%%%%%%%%%%%%%%%%%%%%%%%%%%%%%%%%%%%%%%%%%%%%%%%%%%%%%%%%%%%%%%%%%%%%%%%%%%%%%%%%%%%%%%%%%%%%%%%%%%%%%%%%%%%%%%%%%%%%%%%%%%%%%%%%%%%%%%%%%%%%%%%%
%\mainmatter 

\section{Introduction}

\input{Intro_SyNST}

%\input{history}

%%%%%%%%%%%%%%%%%%%%%%%%%%%%%%%%%%%%%%%%%%%%%%%%%%%%%%%%%%%%%%%%%%%%%%%%%%%%%%%%%%%%%%%%%%%%%%%%%%%%%%%%%%%%%%%%%%%%%%%%%%%%%%%%%%%%%%%%%%%%%%%%
\section{Supports}

\input{Supports_2}

\section{S-Trees}

\input{Strees_B}

%%%%%%%%%%%%%%%%%%%%%%%%%%%%%%%%%%%%%%%%%%%%%%%%%%%%%%%%%%%%%%%%%%%%%%%%%%%%%%%%%%%%%%%%%%%%%%%%%%%%%%%%%%%%%%%%%%%%%%%%%%%%%%%%%%%%%%%%%%%%%%%%
%\section{Graph operations closed over S-trees}

%\input{graph_operation}

%%%%%%%%%%%%%%%%%%%%%%%%%%%%%%%%%%%%%%%%%%%%%%%%%%%%%%%%%%%%%%%%%%%%%%%%%%%%%%%%%%%%%%%%%%%%%%%%%%%%%%%%%%%%%%%%%%%%%%%%%%%%%%%%%%%%%%%%%%%%%%%%
%\section{N-Trees}

%\input{NTrees}

%%%%%%%%%%%%%%%%%%%%%%%%%%%%%%%%%%%%%%%%%%%%%%%%%%%%%%%%%%%%%%%%%%%%%%%%%%%%%%%%%%%%%%%%%%%%%%%%%%%%%%%%%%%%%%%%%%%%%%%%%%%%%%%%%%%%%%%%%%%%%%%%%

\section{Null Decomposition of Trees} \label{NDTSection}

\input{SpectralDecomposition_2}

%%%%%%%%%%%%%%%%%%%%%%%%%%%%%%%%%%%%%%%%%%%%%%%%%%%%%%%%%%%%%%%%%%%%%%%%%%%%%%%%%%%%%%%%%%%%%%%%%%%%%%%%%%%%%%%%%%%%%%%%%%%%%%%%%%%%%%%%%%%%%%%%

%\section{Aplications}

%\input{Aplications}
\section*{Acknowledgement}
	The authors are gratefully indebted to Vilmar Trevisan for their active interest in the publication of this paper. We gratefully acknowledge the many helpful suggestions of Adri\'{a}n Pastine during the preparation of the paper. Even though the text does not reflect it, we carreid on many numerical experiments on \cite{sage}. They gives us the insight for this paper. Aus dem Paradies, das SageMathcloud uns geschaffen, soll uns niemand vertreiben k\"{o}nnen.

%No one will drive us from the paradise which Sagemathcloud created for us.
%{}\\
\section*{}
Funding: This work was partially supported by the Universidad Nacional de San Luis, Grant: PROIPRO 03-2216, and Secretaria de Pol\'{\i}ticas Universitarias, Ministerio de Educaci\'{o}n, Rep\'{u}blica Argentina, Programa Redes Interuniversitarias IX, Grant: \textquotedblleft Red Argentino-Chilena de Teor\'{\i}a de N\'{u}meros, Grafos y Combinatoria\textquotedblright, RESOL-2016-1968-E-APN-SECPU-ME.

\section*{References}

\bibliographystyle{apalike}

\bibliography{TAGcitas}

\end{document}

%% file: Intro_SyNST.tex
The Eigenspaces of graphs have been studied for many years. The standard references in the topic is \cite{cvetkovic1997eigenspaces}. Fiedler  (1975) was the first in studying graph structure with eigenvectors, see \cite{fiedler1975eigenvectors}. In 1988, Power used eigenvectors to study the connection structure of graphs, see \cite{powers1988graph}. The null space has been studied for many classes of graphs (paths, trees, cycles, circulant graphs, hypercubes, etc.). But, compared to the amount of research on spectral graph theory, the study of the eigenvector of graphs has received little attention.
 
The nullity of a tree can be given in an explicit form in terms of the matching number of the tree. In 2005, Fiorini, Gutman, and Sciriha, see \cite{fiorini2005trees}, proved that among all the \(n\)-vertex trees whose vertex degree do not exceed a certain value \(D\), the greatest nullity is \(n-2 \lceil \frac{n-1}{D}\rceil\). They also gave methods for constructing trees with maximum nullity. The work of Fiorini, Gutman and Sciriha is based on the fact that for any tree \(T\) holds \(\nulidad{T}=v(T)-2\nu(T)\), where \(\nu(T)\) is the matching number. This is another consequence of the well-known fact that for trees the characteristic and the matching polynomials are equal.

Sander and Sander (2009) work with ideas similar to ours, see \cite{Sander2009133}, but with different aims. They present a very interesting composition-decomposi-tion technique that correlates tree eigenvectors with certain eigenvectors of an associated skeleton forest (via some contractions). They use the matching properties of a skeleton in order to determine the multiplicity of the corresponding tree eigenvalue. Their results allow them characterizing the tree that admit eigenspaces bases with consisting of vectors whose entries come from \(\{-1,0,1\}\).

The purpose of this study is to determine which information about a trees could be obtained from the support of null space of  its adjacency matrix. We will introduce a new family of trees, the S-trees, which are based on the non-zero entries of vectors in null space. We will show that every tree can be decomposed into a forest of S-trees and a forest of non-singular trees.  

Our work can be seen how a further step (in a narrow sense) of the work of Nylen (\cite{nylen1998null}), and (part of)  work of Neumaier (specifically, section 3 of \cite{neumaier1982second}); even though we were not aware of this former paper before finishing the present work. The null decomposition of trees allow us to note that Theorem 3.4 (ii) and Proposition 3.6 (ii)-(v) in \cite{neumaier1982second} are not correct.

Now we describe as the paper is organized. In Section 2 we set up notation and terminology, and also review some of the standard facts on graphs. Section 3 is concerned with the notion of support of vectors associated to graphs. Section 4 defines and studies S-trees. In Section 5 we state and prove our main result: the null decomposition of trees. We use it in order to obtain formulas for independence number and matching number of a tree. We also prove that the number of maximum matchings in a tree depends on its null decomposition.

\section{Basics and notation}
The material in this section is standard. We recommend that the reader starts reading from Section 3, and comes back to Section 2 only to clear any notation doubts.

As usual in combinatorics, \([k]:=\{1,\cdots,k\}\). In this work we will only consider finite, loopless, simple graphs. Let \(G\) be  a graph:
\begin{enumerate}
	\item \(V(G)\) is the set of vertices of \(G\), and \(v(G):=|V(G)|\) denote its cardinality. An \(n\)-graph is a graph of order \(n\).
	\item For any \(S \subset V(G)\), the subgraph induced by \(S\) in \(G\) is denoted by \(G\langle S \rangle\). 
	\item \(E(G)\) is the set of edges of \(G\), and \(e(G):=|E(G)|\) its size.
	\item Let \(u \sim v\) denote that two vertices \(u\) and \(v\) of \(G\) are neighbors: \(\{u,v\} \in E(G)\).
	\item Let \(N_{G}(v)\) denote set of neighbors of \(v\) in \(G\), if \(G\) is clear from the context we just write \(N(v)\). The closed neighborhood of \(v\) is \(N[v]=N(v)\cup\{v\}\). For \(S \subset V(G) \) the closed neighborhood of \(S\) is
	\[
	N\left[ S \right]:= \bigcup_{u \in S}N\left[u\right]
	\]
	\item Let \(\deg(v)\) denote the degree of \(v\), the cardinality of \(N(v)\).
	\item A vertex \(v\) of \(G\) is a pendant vertex if \(\deg(v)=1\).
	\item Let \(u,v \in V(G)\), with \(G+\{u,v\}\) we denote the graph obtained by add the edge \(\{u,v\}\) to \(E(G)\).
	\item Let \(e \in E(G)\), with \(G-e\) we denote the graph obtained by remove the edge \(e\) from \(G\), thus \(E(G-e)=E(G) \setminus \{e\}\).
	\item With \(\mathbb{R}^{G}\) we denote the vector space of all functions from \(V(G)\) to \(\mathbb{R}\), the set of real numbers. Let \(x \in \mathbb{R}\), and \(v \in V(G)\), we usually write \(x_{v}\) instead of \(x(v)\).
	\item Let \(\theta\) denote the zero vector of a given vector space.
	\item \(A(G)\) is the adjacency matrix of \(G\), if \(G\) is clear from the context we drop \(G\) and just write \(A\).
	\item The rank of \(G\) is the rank of its adjacency matrix: \(\rank{G}:=\rank{A(G)}\). Given a matrix \(A\), its transpose will be denoted by \(A^t\).
	\item The null space of \(G\) is the null space of its adjacency matrix: \(\N{G}:=\N{A(G)}\).
	\item The nullity of \(G\) is the nullity of its adjacency matrix: \(\nulidad{G}:=\nulidad{A(G)}\).
	\item The spectrum of \(G\) is the set of different eigenvalues of \(A(G)\), and it will be denoted by \(\sigma (G)\). Given an eigenvalue \(\lambda \in \sigma (G) \) the eigenspace associate to \(\lambda\), denoted by \(\mathcal{E}_{\lambda}(G)\), will be called \(\lambda\)-eigenspace of \(G\).
	\item Let \(G\) be a graph of order \(n\), and let \(x\) be a vector of \(\mathbb{R}^{n}\). For each vertex \(u \in V(G)\) its \(x\)-neighborhood-weight is
	\[\omega_{x}(u):=\sum_{v \sim u} x_{v}\]
	where \(x_{v}\) is the coordinate of \(x\) associated to the vertex \(v\).
	\item A set \(S\) of vertices of a graph \(G\) is an independent set in \(G\) if no two vertices in \(S\) are adjacent. \(\alpha(G)\) denote the independence number of \(G\), the cardinality of a maximum independent set in \(G\).
	\item A matching \(M\) in  \(G\) is a set of pairwise non-adjacent edges; that is, no two edges in \(M\) share a common vertex. A vertex is saturated (by \(M\)), if it is an endpoint of one of the edges in the matching \(M\). Otherwise the vertex is non-saturated. \(\nu(T)\) is the matching number of \(T\): cardinality of a maximum matching. The set of all maximum matchings of \(G\) is denoted by \(\mathcal{M}(G)\), the number of maximum matchings in \(G\) is \(m(G)=|\mathcal{M}(G)|\). The Edmond-Gallai vertices of \(G\), denoted \(EG(G)\), is the set of all vertices of \(G\) non-saturated by some maximum matching \(M\) in \(G\).
	\item A vertex cover of \(G\) is a set of vertices such that each edge of \(G\) is incident to at least one vertex of the set. The vertex cover number, denoted by \(\tau(G)\), is the size of a minimum vertex cover in \(G\).
	\item A set \(S \subset V(G) \) is a dominating set of \(G\) if each vertex in \(V(G)\) is either in \(S\) or is adjacent to a vertex in \(S\). The domination number \(\gamma (G)\) is the minimun cardinality of a dominating set of \(G\).
	\item A graph \(G\) is bipartite if its vertices can be partitioned in two sets in such a way that no edge join two vertices in the same set.
\end{enumerate}
\begin{theorem}[K\"{o}nig-Egerv\'{a}ry]
	In any bipartite graph \(G\), the number of edges in a maximum matching equals the number of vertices in a minimum vertex cover: \(\nu(G)=\tau(G)\).
\end{theorem}	

The complement of a vertex cover in any graph is an independent set, thus the complement of  a minimum vertex cover is a maximum independent set. Hence
\[
\alpha(G)+\tau(G)=v(G)
\]
The independence number \(\alpha(G)\) of a graph  \(G\) and its domination number \(\gamma(G)\) are related by
\[
\gamma(G) \leq \alpha(G)
\]

% We use the usual notation to work with submatrices, see \cite{bapat2014graphs}: let \(U\) and \(W\) be two subsets of \([n]\), then for an \(n \times n\) matrix \(A\), with \(A(U,W)\) we denote the submatrix of \(A\) obtained by deleting the rows indexed by \(U\) and the columns indexed by \(W\) from \(A\). When \(U=W\), we shorten \(A(U,U)\) to \(A(U)\). %With \(A[U,W]\) denotes the submatrix of \(A\) determined by the rows corresponding to\(U\) and the columns corresponding to \(W\).

%The next result 
%\begin{theorem}[Courant-Fischer]\label{CourantFisher}
%	The eigenvalues \(\lambda_{1}\geq \dots \geq \lambda_{n}\) of a hermitian \(n \times n\) matrix \(A\) are
%	\[
%	\lambda_{i}=\max_{\dim S=i} \min_{\substack{x \in S\\ \|x\|_{2}=1}} x^{*}Ax
%	\] 
%	and, 
%\[
%\lambda_{i}=\min_{\dim S=n-i+1} \max_{\substack{x \in S\\ \|x\|_{2}=1}} x^{*}Ax
%\] 
%where \(S\) is a subspace of \(\mathbb{C}^{n}\).
%\end{theorem}

%% file: Supports_2.tex
%%%%%%%%%%%%%%%%%%%%%%%%%%%%%%%%%%%%%%%%%%%%%%%%%%%

\begin{definition}
Let \(x\) be a vector of \(\mathbb{R}^{n}\), the \textbf{support} of \(x\) is 
\[
\Supp{\mathbb{R}^{n}}{x}:= \lbrace v \in [n] : x_{v} \neq 0 \rbrace 
\]
Let \(S\) be a subset of \(\mathbb{R}^n\). Then the support of \(S\) is
\[
\Supp{\mathbb{R}^{n}}{S}:= \bigcup_{x \in S}  \Supp{\mathbb{R}^{n}}{x}
\]
The cardinality of support of \(x\) is denoted by \(\supp{\mathbb{R}^{n}}{x}:=|\Supp{\mathbb{R}^{n}}{x}| \).
\end{definition}

For example, consider the following set of vectors of \(\mathbb{R}^{6}\):
\[
S= \{(0,1,0,-1,0,0)^{t},(0,0,1,-1,0,0)^{t}\}
\]
Then \(\Supp{\mathbb{R}^{6}}{S} = \{2,3,4\}\), and \(\supp{\mathbb{R}^{6}}{S} = 3\).

%\begin{tabular}{rclcrcl}
%	\(\Supp{\mathbb{R}^{6}}{S,>} \)& \(=\) & \(\{2,3\},\) & {} & \(\supp{\mathbb{R}^{6}}{S,>}\) & \(=\) & \(2\).\\
%	\(\Supp{\mathbb{R}^{6}}{S,<}\) & \(=\)& \(\{4\},\)  & {} & \(\supp{\mathbb{R}^{6}}{S,<}\) & \(=\) & \(1\).\\
%	\(\Supp{\mathbb{R}^{6}}{S,\geq}\) & \(=\) & \(\{1,2,3,5,6\},\) & {} & \(\supp{\mathbb{R}^{6}}{S,\geq}\) & \(=\) &\( 5\).\\
%	\(\Supp{\mathbb{R}^{6}}{S,\leq} \)& \(=\) & \(S,\) & {} & \(\supp{\mathbb{R}^{6}}{S,\leq}\) & = & \(6\).\\
%	\(\Supp{\mathbb{R}^{6}}{S,=}\) & \(=\) & \(\{1,2,3,5,6\}, \)& {} & \(\supp{\mathbb{R}^{6}}{S,=}\) &\(=\) &\( 5\).\\
%		
%\end{tabular}

%\begin{eqnarray*}[rclcrcl]
%
%\end{eqnarray*}

%Given a graph \(G\), and an eigenvalue \(\lambda\) of %\(G\), we usually write \(\Supp{G}{\lambda,*}\) instead of %\(\Supp{G}{\mathcal{E}_{\lambda},*}\).

%The following lemma says that given a graph \(G\), and an eigenvalue \(\lambda\) of \(G\), \(\Supp{G}{\mathcal{E}_{\lambda},\neq}\) can be obtained from any base of the \(\lambda\)-eigenspace \(\mathcal{E}_{\lambda}\).

\medskip

Note that there is a vector in a given subspace that has non-zero entry at a given coordinate if and only if in some (or, every)  basis of the subspace there is a vector in the basis that has a non-zero at the same coordinate.

\begin{lemma} \label{SS1}
	Given a subspace \(S\) of \(\mathbb{R}^{n}\), let \(\mathcal{B} \) a basis of \(S\), then 
	\[
	\Supp{\mathbb{R}^{n}}{S }= \Supp{\mathbb{R}^{n}}{\mathcal{B}}
	\]
\end{lemma}

	%\begin{proof}
	%Clearly, \(\Supp{G}{\mathcal{B},\neq} \subset \Supp{G}{\mathcal{E}_{\lambda},\neq }\). Let \(x \in \mathcal{E}_{\lambda} \). There exists an \(x\), a \( \lambda \)-eigenvector, such that \(x_{v} \neq 0\), then there must exist at a vector in \(\mathcal{B}\) with its \(v\) coordinate non-zero, otherwise all the  \( \lambda \)-eigenvectors will have the \(v\) coordinate null.
	%\end{proof}
	
	%%%%%%%%%%%%%%%%%%%%%%%%%%%%%%%%%%%%%%%%%%%%%%%%
	%
	%%%%%%%%%%%%%%%%%%%%%%%%%%%%%%%%%%%%%%%%%%%%%%%%
%	Given an \(n \times n\) real matrix \(A=[a_{ij}]\), we may define an undirected graph \(\Gamma(A)\) on \(n\) vertices \(1,2,\dots,n\) by including the unordered pair \(\{i,j\}\), i.e., the edge connecting vertex \(i\) to vertex \(j\) in the edge set if and only if \(i \neq j\) and, \(a_{ij} \neq 0\) or \(a_{ji}\neq 0\) . For example the matrix
%	
%	\[
%	A=
%	\left[
%	{ 
%		\begin{array}{rrrrrr}
%		-2 & -1 & 2 & 11 & 0 &  0 \\
%		0  & 3 & 0 & 0 & 0 &  0 \\
%		1  & 0 & 0 & 0 & 0 &  0 \\
%		2  & 0 & 0 & 2 & 0 &  0 \\
%		5  & 0 & 0 & 0 & 1 &  1 \\
%		0  & 0 & 0 & 0 & 2 & 2
%		\end{array}
%	}
%	\right]
%	\]
%	the associated graph \(\Gamma(A)\) is the tree \(T\) given in Figure ~\ref{Fig1}  
%	

	\begin{lemma} \label{SS2}
		Let \(W \) be a subspace of \(\mathbb{R}^{n}\), and let \(S \subset \Supp{\mathbb{R}^{n}}{W}\), then there exist  \(z(S) \in W\), such that \(S \subset \Supp{\mathbb{R}^{n}}{z(S)}\).
	\end{lemma}
	\begin{proof}
		Let \(S=\{i_{1}, \dots,i_{h}\} \subset [n]\). For each \(i_{j} \in S\) there exists a vector \(x(j) \in W\) with non-zero \(j\)-coordinate: \(x(j)_{j} \neq 0\). The following algorithm give a vector \(z(S)\):
		\begin{algorithm}
			INPUT List of vectors \(\{x(1),\dots,x(h)\}\).
			\begin{enumerate}
				\item \(z_{1}=x(1)\).
				\item FOR \(t=2\) TO \(h\):
					\begin{enumerate}
						\item \(\alpha= \max_{1 \leq k \leq n} |x(t-1)_{k}|\).
						\item \(\beta = \frac{1}{2}\min \{|x(t)_{k}|:x(t)_{k}  \neq 0, \text{for } 1 \leq k \leq n \} \).
						\item \(z_{t}=\frac{1}{\alpha} x(t-1) + \frac{1}{\beta} x(t)\).
					\end{enumerate}
			\end{enumerate}
		OUTPUT \(z(S)=z(h)\)
		\end{algorithm}
	Clearly \(\Supp{\mathbb{R}^{n}}{z(t)}=\Supp{\mathbb{R}^{n}}{z(t-1)} \cup \Supp{\mathbb{R}^{n}}{x(t)}\), for \(2 \leq t \leq h\). Hence \(S \subset \bigcup_{t=1}^{h}\Supp{\mathbb{R}^{n}}{x(t)}= \Supp{\mathbb{R}^{n}}{z(S)}\).
	\end{proof}
	Nylen gave the following related result.
	\begin{lemma}[Lemma 7, \cite{nylen1998null}]
		Let \(W\) be a subspace of \(\mathbb{R}^{n}\) with dimension \(d>0\). Then there exists a basis \(\{w_{1},\dots,w_{d}\}\) of \(W\) satisfying 
		\[
		\Supp{\mathbb{R}^{n}}{w_{i}}=\Supp{\mathbb{R}^{n}}{W}
		\]
		for all \(i=1, \dots, d\). 
	\end{lemma}

	Our main interest lies on which properties of a tree are associated to the support of its null. For this reason, we talk about supports of graphs. Given a graph \(G\), and \(x \in \mathbb{R}^{G}\)
	\[
	\Supp{G}{x}:=\{v \in V(G)\; : \; x_{v} \neq 0\}
	\]
	
%	stands for \(\Supp{\mathbb{R}^{6}}{*,*}\), and it will be thought as a subset of vertices of \(T\).
%	
%	For example, the space \(\mathbb{R}^{6}\) can be identify with \(\mathbb{R}^{T}\), where \(T\) is the tree \(T\) given in Figure ~\ref{Fig1}
%	%%%%%%%%%%%%%%%%%%%%%%%%%%%%%%%%%%%%%%%%%%%%%%%%%%%%%%%%%%%
	%
	%%%%%%%%%%%%%%%%%%%%%%%%%%%%%%%%%%%%%%%%%%%%%%%%%%%%%%%%%%%
	
		The adjacency matrix of the tree in Figure \ref{Fig1} is:
		\[
		A(T)=
		\left[
		{ 
			\begin{array}{cccccc}
			0 & 1 & 1 & 1 & 1 & 0 \\
			1 & 0 & 0 & 0 & 0 & 0 \\
			1 & 0 & 0 & 0 & 0 & 0 \\
			1 & 0 & 0 & 0 & 0 & 0 \\
			1 & 0 & 0 & 0 & 0 & 1 \\
			0 & 0 & 0 & 0 & 1 & 0 
			\end{array}
		}
		\right]
		\]
			%%%%%%%%%%%%%%%%%%%%%%%%%%%%%%%%%%%%%%%%%%%%%%%
		%Ejempl0
		
		% Define style for nodes
		\tikzstyle{every node}=[circle, draw, fill=white!,
		inner sep=0.1pt, minimum width=14pt]
		
		\begin{figure}[h] 
			\centering
			
			\begin{tikzpicture}[thick,scale=0.2]%
			\draw 
			(0,0) node{1}
			(0,5) node{2} -- (0,0)
			(-5,0) node{3} -- (0,0)
			(0,-5) node{4} -- (0,0)
			(5,0) node{5} -- (0,0)
			(10,0) node{6} -- (5,0);
			
			\end{tikzpicture}
			\caption{Tree \(T\)}
			\label{Fig1}
		\end{figure}
	
	%%%%%%%%%%%%%%%%%%%%%%%%%%%%%%%%%%%%%%%%%%%%%%%%%%
		The null space of \(T\),  
		\(\N{T}\), is the linear space generated by 
		\[
		\{(0,1,0,-1,0,0)^{t},(0,0,1,-1,0,0)^{t}\}
		\]
		thus
		\[
		\Supp{T}{\N{T}}=\{2,3,4\}
		\]
		
		%We add the proof of the second lemma of \cite{Sander2009133} for reasons of completeness.

		The next two lemmata will be needed in the following section where it will be proved that the null space of trees have structural information. We start with an easy observation about supports of null spaces of a graph \(G\): the neighbors of pendant vertices of a graph \(G\) are not in the support of the null space of \(G\), i.e., the corresponding entries of any 0-eigenvector of \(G\) are zero.

		\begin{lemma} \label{L1}
			Let \(G\) be a \(n\)-graph. If \(v\) is a pendant vertex of \(G\), and \(u\) its neighbor, then \(u \notin  \Supp{G}{\N{G}}\).
		\end{lemma}
		
		\begin{proof}
			If \(  \Supp{G}{\N{G}} = \emptyset\) there is nothing to prove. Let \(x\) be a 0-eigenvector of \(A(G)\). Then, as \(x_{u}=\omega_{x}(v)=\sum_{w \sim v} x_{w} =0 \cdot x_{v}=0\), it follows that \(u \notin  \Supp{G}{\N{G}}\).
		\end{proof}

		Now we turn our attention to null space of trees. Throughout the rest of the work, \(\Supp{}{T} \) stands for \(\Supp{T}{\N{T}}\). Similarly, \(\supp{}{T}\) stands for \(\supp{T}{\N{T}}\).
		
		The next lemma states the impossibility that two supported vertices of \( \Supp{}{T}\) are neighbors. This result is well know (see \cite{neumaier1982second}, pag. 18). But we like our proof.
		
		\begin{lemma} \label{Neighboors of supported}
			Let \(T\) be a tree, and \(v \in  \Supp{}{T}\), then \(N(v) \cap \Supp{}{T} = \emptyset\).
		\end{lemma}
		
		\begin{proof}
Assume to the contrary that there exist \(u_{1}, u_{2} \in \Supp{}{T}\) such that they are neighbors.
Then, by Lemma \ref{SS2} there exists \(x\), a \(0\)-eigenvector of \(T\),  such that \(\Supp{}{T}=\Supp{T}{x}\), i.e., the non-null coordinates of \(x\) are those in \(\Supp{}{T}\), in particular  \(x_{u_{1}} \neq 0\) and \(x_{u_{2}} \neq 0\): 
\[ 
0=0 \cdot x_{v} = \omega_{x}(v)= \sum_{z \in N(v)}x_{z}
\]
but \(u_{1} \in N(u_{2}) \cap  \Supp{}{T}\), then there exists \(u_{3} \in N(u_{2})\) such that \(\sgn{x_{u_{3}}}=-\sgn{x_{u_{1}}}\). Now apply the same argument but with \(u_{2}\) and \(u_{3}\)
in place of \(u_{1}\) and \(u_{2}\) respectively, we obtain a vertex \(u_{4} \in N(u_{3}) \cap \Supp{T}\), different from \(u_{2}\). Note that \(u_{4} \neq u_{1}\); elsewhere we obtain a cycle, which contradicts the assumption that \(T\) is a tree. Continuing in this way, we obtain an infinite path each of whose vertices belongs to \(\Supp{}{T}\). So we arrive at a contradiction.
\end{proof}

%% file: Strees_B.tex
We will prove that all trees are built with two types of bricks. In this section we will introduce the first class of bricks, they are a special kind of singular trees.  
%In the last part of the section we will proved that S-trees are closed under specific graph operations. 

\begin{definition}
A tree \(S\) is an \textbf{S-tree} if \(
N\left[   \Supp{}{S} \right]=V(S) \).
\end{definition}

Let \(S\) be the tree in Figure \ref{Fig2}, we will see that \(S\) is an S-tree, but the tree in Figure \ref{Fig1} is not.

%%%%%%%%%%%%%%%%%%%%%%%%%%%%%%%%%%%%%%%%%%%%%%%%%%%%%%%%%%%%%%%
% Define style for nodes
\tikzstyle{every node}=[circle, draw, fill=white!,
inner sep=0.1pt, minimum width=14pt]

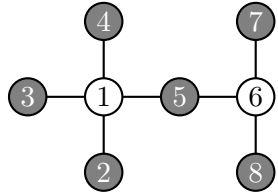
\begin{figure}[h]
	\centering
\begin{tikzpicture}[thick,scale=0.2]%
\draw 
(0,0) node{1}
(-5,0) node[fill=gray!]{\textcolor{white}{3}} -- (0,0)
(0,-5) node[fill=gray!]{\textcolor{white}{2}} -- (0,0)
(0,5) node[fill=gray!]{\textcolor{white}{4}} -- (0,0)
(5,0) node[fill=gray!]{\textcolor{white}{5}} -- (0,0)
(10,0) node{6} -- (5,0)
(10,5) node[fill=gray!]{\textcolor{white}{7}} -- (10,0)
(10,-5) node[fill=gray!]{\textcolor{white}{8}} -- (10,0);

\end{tikzpicture}
\caption{S is an S-tree}
\label{Fig2}
\end{figure}

The adjacency matrix of \(S\) is:
\[
A(S)=
\left[
{ 
	\begin{array}{cccccccc}
	0 & 1 & 1 & 1 & 1 & 0 & 0 & 0\\
	1 & 0 & 0 & 0 & 0 & 0 & 0 & 0\\
	1 & 0 & 0 & 0 & 0 & 0 & 0 & 0\\
	1 & 0 & 0 & 0 & 0 & 0 & 0 & 0\\
	1 & 0 & 0 & 0 & 0 & 1 & 0 & 0\\
	0 & 0 & 0 & 0 & 1 & 0 & 1 & 1\\
	0 & 0 & 0 & 0 & 0 & 1 & 0 & 0\\
	0 & 0 & 0 & 0 & 0 & 1 & 0 & 0\\
	\end{array}
}
\right]
\]
Clearly \(0 \in \sigma (S)\), and the set  
\(
\{(0,1,0,0,-1,0,0,1)^{t}, 
(0,0,1,0,-1,0,0,1)^{t}, \newline (0,0,0,1,-1,0,0,1)^{t},(0,0,0,0,0,0,1,-1)^{t}\}
\)
is a basis of the \(0\)-eigenspace of \(S\).  Thus, by Lemma \ref*{SS1}, \(\Supp{}{S}= \{2,3,4,5,7,8\}
\) and 
\[
N \left[ \Supp{}{S} \right]= N \left[ \{2,3,4,5,7,8\} \right]=V(T)
\]
This prove that \(S\) is an S-tree.
%%%%%%%%%%%%%%%%%%%%%%%%%%%%%%%%%%%%%%%%%%%%%%%%%%%%%%%%%%%%%%%

%In S-trees the non-supported vertices play a structural role.
\begin{definition}
	Let \(S\) be an S-tree. The \textbf{core} of \(S\), denoted by \(\Core{S}\), is defined to be the set of all the neighbors of some supported vertex of \(S\):
	\[
	\Core{S}:=N(\Supp{}{S})
	\]
	A vertex \(v\) is called a \textbf{core-vertex} of \(S\) if \(v \in \Core{S}\). We will denote by  \(\core{S}\) the cardinality of \(\Core{S}\).
\end{definition}

There is just one S-tree without core, the tree of one vertex. Later we will define the notion core for arbitrary trees, see Definition \ref{CoreTreeGeneral}. For S-tree, the core is set of all non-supported vertices:
\begin{lemma}
	Let \(S\) be and S-tree. Then \(\Core{S}=V(T) \setminus \Supp{}{S}\).
\end{lemma}
\begin{proof}
	By Lemma \ref{Neighboors of supported} \(\Supp{}{S}\) is an independent set of \(S\), then \(N(\Supp{}{S}) \cap V(S) =\emptyset\). As \(S\) is an S-tree: \(V(S)=N[\Supp{}{S}]=\Supp{}{S}\cup N(\Supp{}{S})\).
\end{proof}
 The next lemmata are technical but useful. The first one tells that every non-supported vertex of an S-tree has, as neighbors, at least two supported vertices.
\begin{lemma} \label{lemma-number of pendent neighboors}
Let \(S\) be an S-tree. If \(v \in \Core{S}\), then \(
\left| N(v) \cap   \Supp{}{S} \right| \geq 2.
\)
\end{lemma}

\begin{proof}
If \(\left| N(v) \cap   \Supp{}{S} \right|=0\), then \(v \notin N\left[   \Supp{}{S} \right]\), because it has no neighbors in \(  \Supp{}{S}\), but this implies \(S \neq N\left[   \Supp{}{S} \right]\), which is a contradiction because we assume that \(S\) is an S-tree. Then \(\left| N(v) \cap   \Supp{}{S} \right|>0\) for all \(v \in \Core{S}\).

By Lemma \ref{SS2} there exists \(z\), a 0-eigenvector of \(S\), such that \(\Supp{}{S}=\Supp{S}{z}\). Then, as \(\omega_{z}(v)=0\), there are at least two vertices of \(N(v)\) such that its respective coordinates in \(z\) are non-zero.
\end{proof}

Roughly speaking, the next lemma says that the core-vertices of an S-tree fulfills the \text {Hall's condition}, see \cite{jukna2013extremal}. Hence, in any S-tree there are more supported vertices than core vertices.

\begin{lemma}[Hall's condition for S-trees] \label{cardinalities in S-trees}
Let \(S\) be an S-tree. If \(U \subset \Core{S} \), then 
\(
| N(U) \cap  \Supp{}{S}|>|U|
\).
\end{lemma}

\begin{proof}
By induction over the cardinality of \(U\). If \(|U|=1\), then, by Lemma \ref{lemma-number of pendent neighboors}, \(| N(U) \cap  \Supp{}{S}| \geq 2 >1 =|U|\). Assume the lemma holds if \(|U|<k\), with \(k \geq 2\). Let \(U \subset \Core{S}\) with \(|U|=k\). Let \(v\) be a vertex of \(U\). Let \(W_{1},\cdots,W_{t}\) be the trees of the forest \(S \langle (U-v) \cup (N(U-v) \cap \Supp{}{S})\rangle\). Set \(U_{i}=V(W_{i}) \cap U\), for \(i \in [t]\). As \(|U_{i}|<k\), by inductive hypothesis, 
\[
|N_{S}[U_{i}] \cap \Supp{}{S}|\geq |U_{i}|+1
\]
Therefore
\begin{align*}
| N_{S}(U-v) \cap  \Supp{}{S}| & = \left| \left(\bigcup_{i=1}^{t} N_{S}(U_{i})\right)  \cap \Supp{}{S} \right|\\
{} & = \sum_{i=1}^{t} |N_{S}[U_{i}] \cap \Supp{}{S}|\\
{} & \geq \sum_{i=1}^{t} |U_{i}|+1\\
{} & =|U-v|+t
\end{align*}
As \(S\) is a tree, for \(i \in [t]\) holds \(|N_{S}(v) \cap V( W_{i}) \cap \Supp{}{S}| \leq 1\), otherwise we will get a cycle in \(S\). 

If \(| N(v) \cap  \Supp{}{S}| >t\), then \(|\left(N(v) \cap \Supp{}{S}\right) \setminus N(U-v) | \geq 1\). Therefore
\begin{align*}
| N(U) \cap  \Supp{}{S}|& =| N(U-v) \cap  \Supp{}{S}|+|\left(N(v) \cap \Supp{}{S}\right) \setminus N(U-v) |\\ {} & \geq |U-v|+t +1\\ {} & > |U|
\end{align*}
If \(| N(v) \cap  \Supp{}{S}| \leq t\), then \(t \geq 2\). Hence \(|U-v|+t>|U|\). Therefore
\begin{align*}
| N(U) \cap  \Supp{}{S}|  & =| N(U-v) \cap  \Supp{}{S}|+|\left(N(v) \cap \Supp{}{S}\right) \setminus N(U-v) | \\ {} & \geq |U-v|+t \\ {} & > |U|
\end{align*}
\end{proof}

\begin{corollary} \label{coroRefi1}
	Let \(S\) be an S-tree. There exists a matching \(M\) in \(S\) of cardinality \(\core{S}\), and each edge of \(M\) has as incidents vertices a vertex of \(\Core{S}\) and a vertex of \(\Supp{}{S}\). 
\end{corollary}

Recall that \(\mathcal{M}(G)\) is the set of all maximum matchings of a graphs \(G\), and \(m(G):=|\mathcal{M}(G)|\).

\begin{theorem}\label{MatchingMaximumStrees}
	Let \(S\) be an S-tree. Then for all \(M \in \mathcal{M}(S)\), and for all \(e \in M\) we have that \(|e \cap \Core{S}|=1\). 
\end{theorem}

\begin{proof} 
		Let \(M\) be a matching in \(S\) and \(e \in M\). Then, as supported vertices are never neighbors in an S-tree, \(e \cap \Core{S} \neq \emptyset\). If there exists \(e \in M\) such that \(e \subset \Core{S}\), then, by pigeonhole principle, \(|M|<\core{S}\). Therefore, by Corollary \ref{coroRefi1},  \(M \notin \mathcal{M}(S)\).
\end{proof}

\begin{corollary}
	Let \(S\) be an S-tree. If \(v(S)=1\), then  \(\nu(S)=1\), otherwise \(\nu(S)=\core{S}\).
\end{corollary}

From K\"{o}nig-Egerv\'{a}ry Theorem and the previous corollary we deduce that \( \tau(S)=\core{S} \).

\begin{theorem}\label{S-tree-Independent}
	Let \(S\) be an S-tree. Then \(\alpha(S)=\supp{}{S}\) and \(\Supp{}{S}\) is the unique maximum independent set of \(S\). 
\end{theorem}

\begin{proof}
	By Lemma \ref{Neighboors of supported} the supported vertices of \(S\) are independent, therefore \(\alpha(S) \geq \supp{}{S}\). Let \(U \subset V(S)\) be an independent set of \(S\). Write \(U_{1}= U \cap   \Supp{}{S}\) and \(U_{2}= U \setminus S_{1}\). Assume \(U_{2} \neq \emptyset \). As \(U_{2} \subset \Core{S}\). From Lemma \ref{cardinalities in S-trees} there exists a set \(W \subset \Supp{}{S}\) such that \(|W|>|U_{2}|\), and \(W \subset N(U_{2})\). As \(U\) is an independent set of \(S\), it follows that \(W \cap U_{1} = \emptyset\). Hence
	\[
	|U|=|U_{1}|+|U_{2}| < |U_{1}|+|W| \leq  \supp{}{S}
	\]
	Hence \(\Supp{}{S}\) is an independent set of maximum cardinality.
\end{proof}

\begin{theorem} 
	If \(S\) is an S-tree, then \( \tau(S)=\core{S} \). Furthermore \(\Core{S}\) is the unique minimum vertex cover of \(S\).
\end{theorem}

\begin{proof}
As \(\tau(S)+\alpha(S)=v(S)\), we conclude that  \(\tau(S)=v(S)-\alpha(S)=\core{S}\). As the complement of any vertex cover set is an independent set, from Theorem \ref{S-tree-Independent}, we deduce that \(\Core{S}\) is the unique minimum vertex cover of \(S\).
\end{proof}

Let \(M\) be a matching in a graph \(G\), in what follows, \(V(M)\) stands for the set of vertices of \(M\). 

\begin{lemma} \label{SuppNSaturated}
	Let \(S\) be an S-tree. If \(v \in \Supp{}{S}\), then there exists \(M \in \mathcal{M}(S)\) such that \(v \notin V(M)\). 
\end{lemma}

\begin{proof}
	Let \(M\) be a maximum matching. If \(v \notin V(M) \) then there is nothing to prove. Hence assume that \(v \in V(M)\). We will build a maximum matching from \(M\) which do not saturate \(v\). 
	\begin{algorithm}
		Desaturater algorithm:
		\begin{enumerate}
			\item INPUT: An S-tree \(S\), a matching maximum \(M\in \mathcal{M}(S)\), a vertex \(v \in \Supp{}{S} \cap V(M)\).
			\item \(P=v\), the \textbf{desaturater-path}.
			\item \(u=v\).
			\item WHILE \(u \in V(M)\):
				\begin{enumerate}
					\item CHOSE \(c \in V(S)\) such that \(\{u,c\} \in M\).
					\item IF \( \left(N_{S}(c) \cap \Supp{}{S}\right) \setminus V(M) \neq \emptyset\):
					\begin{enumerate}
						\item CHOSE \(u \in \left(N_{S}(c) \cap \Supp{}{S}\right) \setminus V(M) \)
					\end{enumerate}
					ELSE, CHOSE \(u \in \left( N_{S}(c) \cap \Supp{}{S} \right) \setminus {u}\)
					\item \(P=P+\{c,u\}\)
				\end{enumerate}
			\item OUTPUT: the symmetric difference (of edges) between \(M\) and \(P\).
				\[
				\hat{M}=M \vartriangle P
				\]
		\end{enumerate}
	\end{algorithm}
	Clearly \(\hat{M}\) is maximum matching, and \(v \notin V(\hat{M})\).
\end{proof}

\begin{theorem}
	If \(S\) is an S-tree, then \(EG(S)=  \Supp{}{S} \).
\end{theorem}

\begin{proof}
	Just use Theorem \ref{MatchingMaximumStrees} and Lemma \ref{SuppNSaturated}.
\end{proof}

Let \(S\) be an S-tree of order 3 or more, by definition, \(\Core{S}\) is a dominating set of \(S\) (remember that \(\Supp{}{S}\) is an independent set of \(S\)), then \(\gamma (S) \leq \core{S}\). In the S-tree \(S\) in Figure \ref{Fig3}, \(D=\{1,3,5,7,10\}\) is a minimum dominating set, in this case \(\gamma (S)=|D|=5 < 7 =\core{S}\). Even more, no subset of \(\Core{S}=\{1,2,3,4,5,6,7\}\) is a minimum dominating set of \(T\).
%\begin{align*}
%\Core{T} = & \{1,2,3,4,5,6\}\\
%D = & \{1,2,4,5,6\}
%\end{align*}
% Define style for nodes
\tikzstyle{every node}=[circle, draw, fill=white!,
inner sep=0.1pt, minimum width=14pt]

% Define style for nodes
\tikzstyle{every node}=[circle, draw, fill=white!,
inner sep=0.1pt, minimum width=14pt]

\begin{figure}[h]
	\centering
	\begin{tikzpicture}[thick,scale=0.2]%
	
	\draw 
	(0,0) node{$8$}
	(5,0) node[fill=gray]{\textcolor{white}{$1$}} -- (0,0)
	(10,0) node[fill=gray]{\textcolor{white}{$10$}} -- (5,0)
	(15,0) node{$6$} -- (10,0)
	(20,0) node{$15$} -- (15,0)
	(25,0) node[fill=gray]{\textcolor{white}{$7$}} -- (20,0)
	(30,0) node{$16$} -- (25,0)
	
	(5,5) node{$9$} -- (5,0)
	
	(15,5) node{$11$} -- (10,5)
	(10,5) node{$2$} -- (10,0)
	(20,5) node[fill=gray]{\textcolor{white}{$3$}} -- (15,5)
	(25,5) node{$12$} -- (20,5)
	
	(15,-5) node{$13$} -- (10,-5)
	(10,-5) node{$4$} -- (10,0)
	(20,-5) node[fill=gray]{\textcolor{white}{$5$}} -- (15,-5)
	(25,-5) node{$14$} -- (20,-5);
	\end{tikzpicture}
%\end{figure}
	\caption{ S-tree \(S\) and its unique minimum domination set}
	\label{Fig3}
\end{figure}
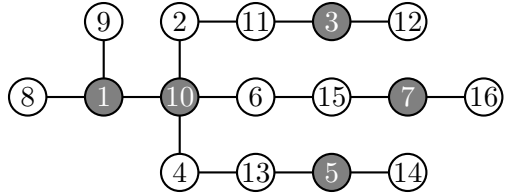

%We conjecture that for every S-tree \(T\) there exists \(D %\subset \Core{T}\), a minimum dominating set of \(T\).

%% file: SpectralDecomposition_2.tex
A tree \(T\) is non-singular tree,  \textbf{N-tree} for short, if its adjacency matrix \(A(T)\) is invertible. Thus \(T\) is a N-tree if and only if \(T\) has a perfect matching, see \cite{bapat2014graphs}. By 	K\"{o}nig-Egerv\'{a}ry Theorem, if \(T\) is a N-tree, \(\alpha(T)=\frac{v(T)}{2}\). The distance between two pendant vertices of a N-tree is always greater than 2. Hence \(l(T) \leq \gamma(T) \leq \frac{v(T)}{2}\), where \(l(T):=|\{v \in V(T): \deg(v)=1 \}|\).

\begin{definition}
Let \(T\) be a tree. The \textbf{S-Set} of T, denoted by \(\mathcal{F}_{S}(T)\), is defined to be the set of connected components of the forest induced by the closed neighbor of \(\Supp{}{T}\) in \(T\):
\[
\mathcal{F}_{S}(T):=\{S \; : \; S \text{ is a connected component of } T\left\langle N\left[ \Supp{}{T} \right] \right\rangle \}
\]
The \textbf{N-set} of T, denoted by \(\mathcal{F}_{N}(T)\), is defined to be the set of connected components of the remaining forest:
\[
\mathcal{F}_{N}(T):= \{N \; :\; N \text{ is a connected component of } T \backslash \mathcal{F}_{S}(T)\}
\]
The pair of sets \((\mathcal{F}_{S}(T),\mathcal{F}_{N}(T))\)  is called the \textbf{null decomposition} of \(T\).
\end{definition}

We think the sets \(\mathcal{F}_{S}(T)\) and \(\mathcal{F}_{N}(T)\) as forests. Thus \(E(\mathcal{F}_{S}(T))\) is the set of edges of all trees in \(\mathcal{F}_{S}(T)\). With \(V(\mathcal{F}_{N}(T))\) we denote the set of vertices of all trees in \(\mathcal{F}_{N}(T)\). Even some times we talk about the S-forest \(\mathcal{F}_{S}(T)\) instead of S-set, and the about the N-forest \(\mathcal{F}_{N}(T)\) instead of N-set.

\begin{definition}\label{CoreTreeGeneral}
	Let \(T\) be a tree, the \textbf{core} of \(T\), denoted by \(\Core{T}\), is the set
	\[
	\Core{T}:=N(\Supp{}{T})
	\]
	As before, \(\core{T}=|\Core{T}|\).
\end{definition}

In the following lemma we collect together some observations about null decomposition of a tree \(T\), for referential issues.
\begin{lemma}
	Let \(T\) be a tree. Then for \(S \in \mathcal{F}_{S}(T)\)
	\begin{enumerate}
		\item The sets \(\Supp{}{T}\), \(\Core{T}\), and \(V(\mathcal{F}_{N}(T)) \) are a partition of \(V(T)\), in a weak sense (some of the sets can be empty sets).
		\item Let \(v \in \Supp{}{T}\), and \(S \in \mathcal{F}_{S}(T)\). If \(v \in e \in E(T)\), then \(e \in E(S)\), i.e. \(N_{S}(v)=N_{T}(v)\).
		\item If \(S \in \mathcal{F}_{S}(T)\), then \(v(S)\geq 3\).
		\item \(N[\Supp{}{T}\cap V(S)]=N[\Supp{}{T}]\cap V(S)=V(S)\).
		\item \(N(\Supp{}{T}\cap V(S))=N(\Supp{}{T})\cap V(S)\).
	\end{enumerate}
\end{lemma}

\begin{proof}
	1, 2 and 3 are direct. For 3, note that \(S\) is a connected component of \(T\left\langle N\left[ \Supp{}{T} \right] \right\rangle\), therefore \(N[\Supp{}{T}]\cap V(S)=V(S)\). Let \(v \in N[\Supp{}{T}\cap V(S)] \). If \(v \in \Supp{}{T}\), then \(u \in N[\Supp{}{T}]\cap V(S)\). If \(v \notin \Supp{}{T}\), then must exist \(u \in \Supp{}{T}\) such that \(u \sim v\). As \(u,v\) are in the same connected component of \(T\left\langle N\left[ \Supp{}{T} \right] \right\rangle\), and \(S\) is the connected component of  \(T\left\langle N\left[ \Supp{}{T} \right] \right\rangle\) where \(v\) is, we conclude that \(u \in V(S)\). Hence \(u \in N[\Supp{}{T}]\cap V(S)\). Thus \(N[\Supp{}{T}\cap V(S)] \subset N[\Supp{}{T}]\cap V(S)\). Similar arguments prove the rest of the statements. 
\end{proof}

For any tree \(T\) there is a set of edges in \(E(T)\) that are not edges neither of any tree in \(\mathcal{F}_{S}(T)\), nor of any tree in \(\mathcal{F}_{N}(T)\). We denoted them by \(\ConnE{T}\):
\[
\ConnE{T}:=E(T) \setminus \left(E(\mathcal{F}_{S}(T)) \cup E(\mathcal{F}_{N}(T))  \right) 
\]
Note that for all \(e \in \ConnE{T}\), we have that \(e \cap V(\Core{T}) \neq \emptyset \) and \(e \cap V(\mathcal{F}_{N}(T)) \neq \emptyset \).

 Given \(S \in \mathcal{F}_{S}(T) \) and \(N \in \mathcal{F}_{N}(T) \) we say that \(S\) and \(N\) are \textbf{adjacent parts}, denoted by \(S \sim N\), if there exists \(e \in \ConnE{T}\) such that  \(e \cap V(S) \neq \emptyset \) and  \(e \cap V(N) \neq \emptyset \).
The set of all the vertices of \(T\) incident to a connection edge is denoted by \(V(\ConnE{T})\).

For example, consider the tree in Figure \ref{fig_el}, its S-set is
\[
\mathcal{F}_{S}(T)=\{T \langle \{1,2,3\}\rangle, T \langle \{4,5,6,7,8\}\rangle, T \langle \{9,10,11,12\}\rangle\}=\{S_{1}, S_{2},S_{3}\}
\]
and its N-Set is
\[
\mathcal{F}_{N}(T)=\{T \langle \{13,14\}\rangle, T \langle \{15,16,17,18 \}\rangle\}=\{N_{1}, N_{2}\}
\]
The connection edges of \(T\) are
\[
\ConnE{T}= \{\{1,13\},\{4,14\},\{9,13\},\{9,16\}\}
\]
and \(V(\ConnE{T})=\{1,4,9,13,14,16\}\). The core of \(T\) is \(\Core{T}=\{1,4,5,9\}\). Note that \(\Core{S_{1}}=\{1\}\), \(\Core{S_{2}}=\{4,5\}\), and \(\Core{S_{3}}=\{9\}\). The support of \(T\) is \(\Supp{}{T}=\{2,3,6,7,8,10,11,12\}\). Note that \(\Supp{}{S_{1}}=\{2,3\}\), \(\Supp{}{S_{2}}=\{6,7,8\}\), and \(\Supp{}{S_{3}}=\{10,11,12\}\).
%Thus \(k_{S}(T)=12\) and \(k_{N}(T)=6\).

%%%%%%%%%%%%%%%%%%%%%%%%%%%%%%%%%%%%%%%%%%%%%%%%%%%%%%%%%%%%%%%%
\begin{figure}[h]
	\centering
	\begin{tikzpicture}[thick,scale=0.2]%
	
	\draw 
	(0,0) node[fill={rgb:black,1;white,3}]{2}
	(5,0) node[fill={rgb:black,1;white,6}]{1} -- (0,0)
	(10,0) node{13} -- (5,0)
	(15,0) node{14} -- (10,0)
	(20,0) node[fill={rgb:black,1;white,6}]{4} -- (15,0)
	(5,5) node[fill={rgb:black,1;white,3}]{3} -- (5,0)
	(15,-5) node{16} -- (10,-5)
	(20,-5) node{15} -- (15,-5)
	(25,0) node[fill={rgb:black,1;white,3}]{7} -- (20,0)
	(25,-5) node[fill={rgb:black,1;white,6}]{5} -- (25,0)
	(10,-5) node[fill={rgb:black,1;white,6}]{9} -- (10,0)
	(5,-5) node[fill={rgb:black,1;white,3}]{10} -- (10,-5)
	(10,-10) node[fill={rgb:black,1;white,3}]{12} -- (10,-5)
	(15,-15) node{18} -- (15,-10)
	(15,-10) node{17} -- (15,-5)
	(25,-10) node[fill={rgb:black,1;white,3}]{8} -- (25,-5)
	(20,5) node[fill={rgb:black,1;white,3}]{6} -- (20,0)
	(10-3.5355,-8.5355) node[fill={rgb:black,1;white,3}]{11} -- (10,-5);

	\draw [densely dashdotted,very thick] (0,-2) -- (4.8,-2);
	\draw [densely dashdotted,very thick] (5,-2) arc (-90:0:2);
	\draw [densely dashdotted,very thick] (0,2) arc (90:270:2);
	\draw [densely dashdotted,very thick] (7,5) -- (7,0);
	\draw [densely dashdotted,very thick] (3,4.9) -- (3,4.3);
	\draw [densely dashdotted,very thick] (0.3,2) -- (1,2);
	\draw [densely dashdotted,very thick] (1,2) arc (-90:0:2);
	\draw [densely dashdotted,very thick] (3,5) arc (180:0:2);
	
	\draw [densely dashdotted,very thick] (27,0) -- (27,-10);
	\draw [densely dashdotted,very thick] (27,-10) arc (0:-180:2);
	\draw [densely dashdotted,very thick] (27,0) arc (0:90:2);
	\draw [densely dashdotted,very thick] (23,-10) -- (23,-4);
	\draw [densely dashdotted, very thick] (27,0) -- (27,-10);
	\draw [densely dashdotted, very thick] (22,5) arc (0:180:2);
	\draw [densely dashdotted, very thick] (20,-2) arc (-90:-180:2);
	\draw [densely dashdotted, very thick] (23,-4) arc (0:90:2);
	\draw [densely dashdotted, very thick] (24,2) arc (-90:-180:2);
	\draw [densely dashdotted, very thick] (18,0) -- (18,5);
	\draw [densely dashdotted, very thick] (20,-2) -- (21,-2);
	\draw [densely dashdotted, very thick] (25,2) -- (24,2);
	\draw [densely dashdotted, very thick] (22,5) -- (22,4);
	
	\draw [densely dashdotted, very thick] (10,-3) -- (5,-3);
	\draw [densely dashdotted, very thick] (12,-5) -- (12,-10);
	\draw [densely dashdotted, very thick] (12,-5) arc (0:90:2);
	\draw [densely dashdotted, very thick] (5,-3) arc (90:270:2);
	\draw [densely dashdotted, very thick] (12,-10) arc (0:-180:2);
	\draw [densely dashdotted, very thick] (5,-7) arc (135:315:2);
	
	\draw (0,5) node[white]{\textcolor{black}{$S_1$}};
	\draw (25,5) node[white]{\textcolor{black}{$S_2$}};
	\draw (2.5,-10) node[white]{\textcolor{black}{$S_3$}};
	
	\draw [dotted] (10,-2) -- (15,-2);
	%\draw [dotted] (8,-2) -- (8,2);
	\draw [dotted] (10,2) -- (15,2);
	%\draw [dotted] (17,2) -- (17,-2);
	\draw [dotted] (15,2) arc (90:-90:2);
	\draw [dotted] (10,-2) arc (-90:-270:2);
	
	\draw [dotted] (15,-3) -- (20,-3);
	%\draw [dotted] (22,-3) -- (22,-7);
	%\draw [dotted] (22,-7) -- (17,-7);
	\draw [dotted] (17,-9) -- (17,-15);
	%\draw [dotted] (17,-17) -- (13,-17);
	\draw [dotted] (13,-15) -- (13,-5);
	\draw [dotted] (15,-3) arc (90:180:2);
	\draw [dotted] (20,-3) arc (90:-90:2);
	\draw [dotted] (17,-15) arc (0:-180:2);
	\draw [dotted] (19,-7) arc (90:180:2);
	\draw [dotted] (20,-7) -- (19,-7);

	\draw (12.5,4) node[white]{\textcolor{black}{$N_1$}};
	\draw (19,-12.5) node[white]{\textcolor{black}{$N_2$}};

	\end{tikzpicture}
	\caption{Null decomposition of a tree \(T\)}
	\label{fig_el}
\end{figure}
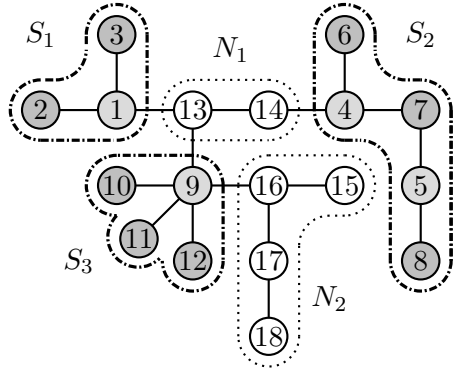
%%%%%%%%%%%%%%%%%%%%%%%%%%%%%%%%%%%%%%%%%%%%%%%%%%%%%%%%%%%%%%%%

Given a graph \(G\), let \(x\) be a vector of \(\mathbb{R}^{G}\). Let \(H\) be a subgraph of \(G\). Then \(\down{x}{G}{H}\) is the vector that we obtain when we restrict \(x\) to the coordinates (vertices) associated to \(H\). For example, consider the tree \(T\) in Figure ~\ref*{fig_el}. Let \(
x^{t}=(2,4,\dots,36) \in \mathbb{R}^{T}\). Then  
\[
(\down{x}{T}{$S_{2}$})^{t} = (8,10,12,14,16)
\]
For any \(n\)-graph \(G\) and any \(m\)-induced-sugbraph \(H \leq G\), given \(x \in \mathbb{R}^{H}\), whose coordinates are indexed by \(V(H)\), with  \(\up{x}{G}{H}\) we write the lift of \(x\) to a vector of \(\mathbb{R}^{G}\) in the following way
\begin{itemize}
	\item For any \(u \in V(G) \setminus V(H)\), \((\up{x}{G}{H})_{u}=0\).
	\item For any \(u \in V(H)\), \((\up{x}{G}{H})_{u}=x_{u}\).
\end{itemize}
Consider the tree \(T\) and the subtree \(S_{3}\), see Figure ~\ref{fig_el}. Let \( x^{t}=(1,2,3,4) \in \mathbb{R}^{S_{3}} \). Then   
\[
(\up{x}{T}{$S_{3}$})^{t} = (0,0,0,0,0,0,0,0,0,1,2,3,4,0,0,0,0,0,0) \in \mathbb{R}^{T}
\]

\begin{theorem} \label{T_S_forest}
	Let \(T\) be a tree. The S-set \(\mathcal{F}_{S}(T)\) of \(T\) is or an empty set,  or a set of S-trees.
\end{theorem}

\begin{proof}
	Assume \(\mathcal{F}_{M}\neq \emptyset\).  For any \(x \in \N{T}\), if \(v \in \Core{T} \cup V(\mathcal{F}_{N}(T)) \), then \(x_{v}=0\). 
	
	Claim: If \(x \in \N{T}\), then \(A(S) \, \down{x}{T}{S}=0\) for all \(S \in \mathcal{F}_{S}(T)\).
	
	Proof of the claim: Let \(S \in \mathcal{F}_{S}(T)\), and \(v \in V(S)\), as \(\left(A(T)\right))_{vz}=0\) for all \(z \in V(\mathcal{F}_{S}(T)) \setminus V(S)\), and \(x_{z}=0\) for all \(z \in V(\mathcal{F}_{N}(T)\) 
	\begin{align*}
	\left((A(S) \, \down{x}{T}{S} \right)_{v} & = \sum_{\substack{w \sim v \\ w \in V(S)}} \left( \down{x}{T}{S} \right)_{w} \\
		{} & = \sum_{\substack{w \sim v\\ w \in V(S)}} \left( x \right)_{w} + \sum_{\substack{z \sim v \\ z \in V(T)\setminus V(S)}} \left( A(T) \right)_{vz} \, x_{z}\\
		{} & = \left( A(T) \,x \right)\\
		{} & = \theta
	\end{align*}
	Therefore for all \(S \in \mathcal{F}_{S}(T)\), if \(x \in \N{T}\), then \(\down{x}{T}{S} \in  \N{S}\). Hence \(\Supp{}{T} \cap V(S) \subset \Supp{}{S}\). Therefore \(V(S)= N[\Supp{}{T}\cap V(S)] \subset N\left[ \Supp{}{S}\right]\subset V(S) \). Thus \(S\) is a S-tree.
\end{proof}

\begin{corollary}
Let \(T\) be a tree, and  \(S \in \mathcal{F}_{S}(T)\). Then 
\begin{enumerate}
	\item \(\Supp{}{S}=\Supp{}{T} \cap V(S)\).
	\item \(\Supp{}{T}=\bigcup_{S \in \mathcal{F}_{S}(T)} \Supp{}{S}\).
	\item \(\Core{S}=\Core{T} \cap V(S)\).
	\item \(\Core{T}= \bigcup_{S \in \mathcal{F}_{S}(T)} \Core{S}\).
	\item If \(x \in \N{T}\), then \(\down{x}{T}{S} \in \N{S}\).
	\item If \(x \in \N{S}\), then \(\up{x}{T}{S} \in \N{T}\).
	\end{enumerate}
\end{corollary}

\begin{proof}
	1. Note that from the proof of Theorem \ref{T_S_forest}, we know that \(\Supp{}{T} \cap V(S) \subset \Supp{}{S} \). Assume that exists \(v \in \Supp{}{S} \setminus \Supp{}{T}\). As \(v \in  \Supp{}{S}\subset V(S)\) must exist  \(u \in \Supp{}{T}\) such that \(u \sim v\). As \(u,v\) are in the same connected component of \(T\left\langle N\left[ \Supp{}{T} \right] \right\rangle\),  \(u,v \in V(S)\). This implies that \(u \in \Supp{}{S}\), thus we arrive at a contradiction, because \(\Supp{}{S}\) is an independent set of \(S\).
	
	2. Just note that
	\begin{align*}
	\bigcup_{S \in \mathcal{F}_{S}(T)} \Supp{}{S} & = \bigcup_{S \in \mathcal{F}_{S}(T)} \left( \Supp{}{T} \cap V(S) \right)\\
	{} & = \Supp{}{T} \cap \bigcup_{S \in \mathcal{F}_{S}(T)} V(S)\\
	{} & = \Supp{}{T}
	\end{align*}
	
	3. \(\Core{S}=N(\Supp{}{S})=N(\Supp{}{T}\cap V(T))=N(\Supp{}{T}) \cap V(T)= \Core{T} \cap V(S)\).
	
	4. Just note that 
	\begin{align*}
	\bigcup_{S \in \mathcal{F}_{S}(T)} \Core{T} & = \bigcup_{S \in \mathcal{F}_{S}(T)} \left( \Core{T}\cap V(S) \right)\\
	{} & = \Core{T} \cap \bigcup_{S \in \mathcal{F}_{S}(T)} V(S)\\
	{} & = \Core{T}
	\end{align*}
	
	5. See proof of Theorem \ref{T_S_forest}.
	
	6. Let \(v \in V(T)\)
	\begin{align*}
	\left( A(T) \, \up{x}{T}{S} \right)_{v} & = \sum_{\substack{u \sim v \\ u \in V(T)}} \left( \up{x}{T}{S} \right)_{u}\\
	{} & = \sum_{\substack{u \sim v \\ u \in V(S)}} \left( x \right)_{u}\\
	{}& = \left\lbrace
	\begin{array}{ll}
	0 & \text{if } v \in V(S),\\
	0 & \text{if } v \notin V(S), \text{ and } N_{T}(v) \cap V(S)= \emptyset,\\
	x_{w} & \text{if } v \notin V(S), \text{ and } N_{T}(v) \cap V(S)= \{w\}.
	\end{array}
	\right.
	\end{align*}
If \(w \in V(S)\) and \(N_{S}(w) \neq N_{T}(w)\), then \(w \notin \Supp{}{S}\). Hence \(x_{w}=0\). Therefore \(A(T) \, \up{x}{T}{S} = \theta\).
\end{proof}

\begin{corollary}
	Let \(T\) be a tree. Then
	\begin{enumerate}
		\item Let \(S_{1},S_{2} \in \mathcal{F}_{S}(T)\), \(x_{1} \in \N{S_{1}}\), and \(x_{2} \in \N{S_{2}}\), then \[(\up{x_{1}}{T}{S{\small 1}})^{t} \up{x_{2}}{T}{S{\small 2}}=0\]
		\item For all \(S \in \mathcal{F}_{S}(T)\), and for all \(x(S) \in \N{S}\)
		\[
		\sum_{S \in \mathcal{F}_{S}(T)} \up{x(S)}{T}{S}
		\]
		is a vector of \(\N{T}\).
	\end{enumerate}
\end{corollary}

Let \(H \leq F \leq G\) , two subgraph of a graph \(G\). Let \(U \subset \mathbb{R}^{H}\). Then \(\up{U}{F}{H}:=\{\up{x}{F}{H} \; :\; x \in U\}\).  

\begin{proposition} \label{C_null}
	Let \(T\) be a tree. Then 
	\[
	\N{T}=\bigoplus_{S \in \mathcal{F}_{S}(T)} \up{\N{S}}{T}{S}
	\]
	Therefore \(\nulidad{T}=\sum_{S \in \mathcal{F}_{S}(T)} \nulidad{S}\).
\end{proposition}

\begin{proof}
	Note that we can obtain a base of \(\N{T}\) by taken a base for each \(S \in \mathcal{F}_{S}(T)\), i.e., let 
	\[
	\mathcal{B}_{S}:=\lbrace b(S)_{1},\dots,b(S)_{n(S)} \rbrace
	\]
	be a base of \(\N{S}\), for each \(S \in \mathcal{F}_{S}(T)\). We take the lift of each one:
	\[
	\up{\mathcal{B}_{S}}{T}{S}:=\lbrace \up{b(S)_{1}}{T}{S},\dots,\up{b(S)_{n(S)}}{T}{S} \rbrace\
	\]
	Then
	\[
	\mathcal{B}(T):= \bigcup_{S \in \mathcal{F}_{S}(T)}	\up{\mathcal{B}_{S}}{T}{S}
	\]
	is a base of \(\N{T}\).
\end{proof}

Thus, for the tree \(T\) in Figure \ref{fig_el}, we have that \(\N{S_{1}}=\langle\{(0,-1,1)^t\}\rangle\), \(\N{S_{2}}=\langle\{(0,0,1,-1,1)^t\}\rangle\), and \(\N{S_{3}}=\langle \{(0,1,-1,0)^t,(0,1,0,-1)^t\}\rangle\). Therefore \(\N{T}\) is spanning by \(\{e_{2}-e_{3}, e_{6}-e_{7}+e_{8}, e_{10}-e_{11}, e_{10}-e_{12}\}\).

%\small{
%\[
%\left\{ 
%\left[
%\begin{array}{r}
% 0\\ %1
% 1\\ %2
%-1\\ %3
% 0\\ %4
% 0\\ %5
% 0\\ %6
% 0\\ %7
% 0\\ %8
% 0\\ %9
% 0\\ %10
% 0\\ %11
% 0\\ %12
% 0\\ %13
% 0\\ %14
% 0\\ %15
% 0\\%16
% 0\\%17
% 0 %18
%\end{array}
%\right],
%%%%%%%%%%%%%%%%%%%%%
%\left[
%\begin{array}{r}
%0\\ %1
%0\\ %2
%0\\ %3
%0\\ %4
%0\\ %5
%1\\ %6
%-1\\ %7
%1\\ %8
%0\\ %9
%0\\ %10
%0\\ %11
%0\\ %12
%0\\ %13
%0\\ %14
%0\\ %15
%0\\%16
%0\\%17
%0 %18
%\end{array}
%\right],
%%%%%%%%%%%%%%%%%%%%%%%%%%
%%%%%%%%%%%%%%%%%%%%%
%\left[
%\begin{array}{r}
%0\\ %1
%0\\ %2
%0\\ %3
%0\\ %4
%0\\ %5
%0\\ %6
%0\\ %7
%0\\ %8
%0\\ %9
%1\\ %10
%-1\\ %11
%0\\ %12
%0\\ %13
%0\\ %14
%0\\ %15
%0\\%16
%0\\%17
%0 %18
%\end{array}
%\right],
%%%%%%%%%%%%%%%%%%%%%%%%%%
%%%%%%%%%%%%%%%%%%%%%
%\left[
%\begin{array}{r}
%0\\ %1
%0\\ %2
%0\\ %3
%0\\ %4
%0\\ %5
%0\\ %6
%0\\ %7
%0\\ %8
%0\\ %9
%1\\ %10
%0\\ %11
%-1\\ %12
%0\\ %13
%0\\ %14
%0\\ %15
%0\\%16
%0\\%17
%0 %18
%\end{array}
%\right]
%%%%%%%%%%%%%%%%%%%%%%%%%%
%\right\}
%\]
%}
Next lemma builds, from a maximum matching in \(T\), a maximum matching which does not use any connection edges.

\begin{lemma}\label{L_MNCE}
	Let \(T\) be a tree. There exists a maximum matching \(M\) in \(T\) such that \(M \cap \ConnE{T} = \emptyset\).
\end{lemma}

\begin{proof}
	Let \(M\) be a maximum matching of \(T\) such that \(M \cap \ConnE{T} \neq \emptyset \). The following algorithm give us a new maximum matching \(\tilde{M}\) such that \(|\tilde{M} \cap \ConnE{T}| < |M \cap \ConnE{T}|\). 
	\begin{algorithm} \label{S_matchin_algorithm}
		S-Matching algorithm. 
		\begin{enumerate}
			\item INPUT: \(M \in \mathcal{M}(T)\) , an edge \(e \in M \cap \ConnE{T}\), and \(S \in \mathcal{F}_{S}(T)\) such that \(e \cap V(S) \neq \emptyset \)
			\item \(i=0\).
			\item \(u_{i} \in e \cap V(S) \).
			\item WHILE \(N_{S}(u_{i})\setminus V(M) = \emptyset\):
				\begin{enumerate}
					\item \(i=i+1\).
					\item CHOSE \(v_{i} \in N_{S}(u_{i})\).
					\item CHOSE \(u_{i} \in V(S)\) such that \( \{v_{i},u_{i}\} \in M\).
				\end{enumerate}
			\item \(i=i+1\).
			\item CHOSE \(v_{i} \in (N_{S}(u_{i-1})\setminus V(M)\).
			\item OUTPUT
			\[
			\tilde{M}:=M \setminus \left(e \bigcup_{k=1}^{i-1} \{v_{i},u_{i}\} \right) \cup \bigcup_{k=1}^{i}\{u_{i-1},v_{i}\}
			\]
		\end{enumerate}
	\end{algorithm}
	Repeated applications of the S-Matching algorithm prove the lemma.
\end{proof}

Let \(T\) be a tree. For any \(e \in \ConnE{T}\), we denote \(S_{e} \in \mathcal{F}_{S}(T)\) and \(N_{e} \in \mathcal{F}_{N}(T)\) to the parts of \(T\) connected by \(e\). Let \(u \in \Core{S_{e}}\) and \(v \in V(N_{e})\), we denoted by \(T(e,u,v)\) the tree obtained from \(T\) by replacing \(e\) for \(\lbrace u,v \rbrace\) in \(E(T)\): \(T(e,u,v)=T-e+\{u,v\}\).

\begin{lemma} \label{stability}
	Let \(T\) be a tree \(T\), and \(e \in \ConnE{T}\). For each \(u \in \Core{S_{e}}\) and for each \(v \in V(N_{e})\) we have that
	\begin{enumerate}
		\item If \(M \in \mathcal{M}(T)\) and \(M \cap \ConnE{T}=\emptyset\), then \(M \in M(T(e,u,v))\).
		\item \(\N{T}=\N{T(e,u,v)}\).
		\item \(\nulidad{T}=\nulidad{T(e,u,v)}\).
		\item \(\rank{T}=\rank{T(e,u,v)}\).
		%\item \(\(\mathcal{M}(T)=\mathcal{M}(T(e,u,v))\).
	\end{enumerate}
\end{lemma}

\begin{proof}
	Any maximum matching in \(T\) whose intersection with \(\ConnE{T}\) is empty is a matching of \(T(e,u,v)\). This matching exists by Lemma \ref{L_MNCE}.  From \cite{bevis1995ranks} we know that the rank of the adjacency matrix of a tree equals two times the its matching number. Then \(\rank{T(e,u,v)} = 2\nu(T(e,u,v)) \geq 2|M|=\rank{T}\). On the other hand, \(\N{T} \subset \N{T(e,u,v)}\), thus \(\nulidad{T(e,u,v)} \geq \nulidad{T}\). Therefore
	\begin{align*}
	v(T) & = v(T(e,u,v) \\
	{} & = \rank{T(e,u,v)}+\nulidad{T(e,u,v)}\\
	{} & \geq \rank{T}+\nulidad{T}\\
	{} & = v(T)
	\end{align*}
	Thus \(\nulidad{T}=\nulidad{T(e,u,v)}\), and \(\rank{T}=\rank{T(e,u,v)}\). Hence \(M \in M(T(e,u,v)\) and \(\N{T}=\N{T(e,u,v)}\).
\end{proof}

Let \(T\) be a tree, and \(u,v \in V(T)\), here and subsequently, \(T(u \rightarrow v)\) stands for the following subtree  of \(T\):
\[
T(u \rightarrow v) := T\left\langle \{ x \in V(T): v \in V(uP_{T}x) \} \right\rangle 
\]
See Figure \ref{FigT100}.
%%%%%%%%%%%%%%%%%%%%%%%%%%%%%%%%%%%%%%%%%%%%%%%%%%%%%%
%
% Define style for nodes
\tikzstyle{every node}=[circle, draw, fill=white!,
inner sep=0.1pt, minimum width=11pt]

\begin{figure}[h] 
	\centering
	\begin{tikzpicture}[thick,scale=0.2]%
	\draw
	(0,0) node{}
	(0,5) node[fill=gray]{$v$} -- (0,0)
	(0,10) node{} -- (0,5)
	(-5,10) node{} -- (0,10)
	(0,15) node{} -- (0,10)
	(-5,15) node{} -- (0,15)
	(0,20) node{} -- (0,15)
	(-5,20) node{} -- (0,20)
	(5,15) node{} -- (0,15)
	(10,15) node[fill=gray]{$u$} -- (5,15)
	(15,15) node{} -- (10,15)
	(20,15) node{} -- (15,15)
	(10,10) node{} -- (10,15)
	(15,10) node{} -- (15,15)
	(15,5) node{} -- (15,10);
	
	\draw [dotted] (10,17) -- (20,17);
	\draw [dotted] (8,15) -- (8,10);
	\draw [dotted] (17,11.2) -- (17,5);
	\draw [dotted] (20,13) -- (18.8,13);
	\draw [dotted] (10,8) -- (10.8,8);
	\draw [dotted] (13,5) -- (13,5.8);

	\draw [dotted] (20,13) arc (-90:90:2);
	\draw [dotted] (10,17) arc (90:180:2);
	\draw [dotted] (13,5) arc (180:360:2);
	\draw [dotted] (17,11) arc (180:90:2);
	\draw [dotted] (8,10) arc (180:270:2);
	\draw [dotted] (13,6) arc (0:90:2);
	
	\draw (15,22) node[white]{\textcolor{black}{$T(v \to u)$}};
	
	\end{tikzpicture}
	\caption{\(T\) and \(T(v \to u)\)}
\end{figure}\label{FigT100}

\begin{definition}
	Let \(S\) be an S-tree, and \(c \in \Core{S}\). Given \(v\), a vertex, \(v \notin V(S)\), with \(S +_{s} \{c,v\})\) we denote the tree with vertex set \(V(S +_{s} \{c,v\})= V(S) \cup \{v\}\), and edge set \(E(S +_{s} \{c,v\})=E(S) \cup \{c,v\}\). 
\end{definition}

Usually, we do not care about to which core-vertex the new vertex \(v\) is added, in this cases we just write \(S+_{s}v\). See Figure 6.

% Define style for nodes
\tikzstyle{every node}=[circle, draw, fill=white!,
inner sep=0.1pt, minimum width=11pt]

\begin{figure}
	\centering
	\begin{tikzpicture}[thick,scale=0.2]%
	
	\draw[dotted,very thick]
	(10,5) -- (10,0);
	
	\draw
	(0,0) node[fill=gray]{}
	(10,5) node{v}
	(0,5) node{} -- (0,0)
	(0,-5) node{} -- (0,0)
	(5,0) node{} -- (0,0)
	(10,0) node[fill=gray]{} -- (0,0)
	(15,0) node{} -- (0,0);
	\end{tikzpicture}
	\caption{$S+_S v$}
\end{figure}

\begin{lemma}
	Let \(S\) be an S-tree, and \(v\) a vertex such that \(v \notin V(S)\). Then \(S+_{s}v\) is an S-tree.
\end{lemma}

\begin{proof}
	Clearly \(\up{\N{S}}{$S+_{s}v$}{S} \subset \N{S+_{s}v}\). Let \(x \in \mathbb{R}^{S}\) such that \(\Supp{S}{x}=\Supp{}{S}\). Let \(u \in N_{S}(c)\), where \(c \in \Core{S}\) is the vertex of \(S\) that form the new edge \(\{c,v\}\), and \(\alpha \in (0,1)\). Define
	\(y \in \mathbb{R}^{S+_{s}v} \) as
	\[
	y_{w}:= \left\lbrace 
	\begin{array}{rl}
	\alpha\,x_{w} & \text{if } w \in V(S(c\rightarrow u)),\\
	(1-\alpha)x_{u} & \text{if } w=v,\\
	x_{w} & \text{otherwise.}
	\end{array}
	\right.
	\]
	Clearly \(A(S+_{s}v)\, y=\theta\). Hence \(S+_{s}v\) is an S-tree.
\end{proof}

\begin{theorem}
Let \(T\) be a tree, the N-set \(\mathcal{F}_{N}(T)\) of \(T\) is or an empty set, or a set of N-trees.
\end{theorem}

\begin{proof}
	Assume there exists \(N \in \mathcal{F}_{N}(T)\)  such that \(N\) is a singular tree. Then \(\mathcal{F}_{S}(N) \neq \emptyset \).  Let \(S_{0} \in \mathcal{F}_{S}(N)\), an S-tree of \(N\). Note that \(\Supp{}{S_{0}} \cap \Supp{}{T}=\emptyset\). Let \(x_{0} \in \N{N}\) such that \(\Supp{}{S_{0}}=\Supp{N}{x_{0}} \). If \(N_{T}(S_{0}) \subset N\), then \(\up{x_{0}}{T}{N} \in \N{T}\). This implies that \(\Supp{}{S_{0}}  \subset \cap \Supp{}{T}\), which is a contradiction. Let \(S_{1}, \dots, S_{k}\) the S-trees of \(T\) adjacent to \(S_{0}\) via supported vertices of \(S_{0}\). For each \(i \in [k]\), let \(u_{i} \in \Supp{}{S_{0}}\) the vertex of \(S_{0}\) neighbor of some vertex of \(S_{i}\). Let \(x_{i} \in \N{S_{i}+_{s}u_{i}}\) such that \(\Supp{}{S_{i}+_{s}u_{i}}=\Supp{S_{i}+_{s}u_{i}}{x_{i}}\). Let \(\varrho_{i}=\frac{(x_{0})_{u_{i}}}{(x_{i})_{u_{i}}}\). Then	
	\[
	\up{x_{0}}{T}{S{\small 0}}+\sum_{i=1}^{k} \varrho_{i} \;\up{(\down{x_{i}}{$S_{i}+_{s}u_{i}$}{$S_{i}$})}{T}{$S_{i}+_{s}u_{i}$}
	\]
	is a vector of \(\N{T}\), hence \(S_{0} \subset \Supp{}{T}\), which is a contradiction.	
\end{proof}

This theorem can also be proved by using Lemma \ref{stability}. Here a sketch of the proof. As in the proof, let \(S_{0}\) be an S-part of some tree in \(\mathcal{F}_{N}(T)\). Let \(u \in \Core{S_{0}}\).  We change all the connection edges from \(S_{0}\) to some S-part of \(T\), by replacing its \(S_{0}\) vertices by \(v\). By Lemma \ref{stability}, this new tree has the same null space as \(T\), but in this new tree \(\Supp{}{S_{0}}\) is a subset of its support, which is a contradiction.

Now we have three corollaries, the first one tells that every connection edge is never in a maximum matching. 
\begin{corollary}\label{CoroMatching}
	Let  \(T\) be a tree, and a maximum matching \(M \in \mathcal{M}(T)\). Then \(M \cap \ConnE{T} = \emptyset \).
\end{corollary}

\begin{proof}
	Let \(M \in \mathcal{M}(T)\) such that \(M \cap \ConnE{T} \neq \emptyset\). Repeated applications of the S-Matching algorithm \ref{S_matchin_algorithm} give a matching \(\hat{M} \in \mathcal{M}(T)\) such that \(\hat{M} \cap \ConnE{T} = \emptyset\) and \(e(M)=e(\hat{M})\). Note that \(\hat{M}\) do not uses all vertices in \(\mathcal{F}_{N}(T)\). Let \(M(N)\) be the perfect matching in \(N\), with \(N \in \mathcal{F}_{N}(T)\). Then
	\[
	\bar{M}=\hat{M}-(M\cap E(\mathcal{F}_{N}(T)))+\bigcup_{N \in \mathcal{F}_{N}(T)} M(N)
	\]
	is a matching of \(T\) such that \(e(M)<e(\bar{M})\), which is a contradiction.
\end{proof}

Our next result gives the matching number and the independence number of a tree in terms its core, its support, and its N-forest.

\begin{corollary}
	Let \(T\) be a tree. Then
	\begin{align*}
	\nu(T) = & \core{T}	+\frac{v(\mathcal{F}_{N}(T))}{2}\\
	\alpha(T)= & \supp{}{T}+\frac{v(\mathcal{F}_{N}(T))}{2}
	\end{align*}
\end{corollary}

\begin{proof}
	By Corolary ~\ref{CoroMatching}, any maximum matching in \(T\) must be maximum matching in each part of the null decomposition of \(T\). Then \(\nu(T) = \core{T}	+\frac{v(\mathcal{F}_{N}(T))}{2}\). Hence, by K\"{o}nig-Ergerv\'{a}ry Theorem \(\alpha(T)= \supp{}{T}+\frac{v(\mathcal{F}_{N}(T))}{2}\).
\end{proof}

The tree \(T\) in Figure \ref{fig_el} has matching number \(\nu(T)=4 +\frac{6}{2}=7\), and independence number \(\alpha(T)=8+\frac{6}{2}=11\).

Since 1964, it is known that for trees, the characteristic and the matching polynomials are the same, see \cite{sachs1964beziehungen}. Therefore \(m(T)\), the number of maximum matching of the tree \(T\), equals the product of all nonzero eigenvalues of \(T\). The matching polynomial of a graph had been study in many other papers, see for example \cite{farrell1979introduction}, \cite{godsil1981theory}, and  \cite{godsil1995algebraic}. In 1981, Godsil proved that the characteristic and the matching polynomial are equal if and only if the graph is a forest, see \cite{godsil1981theory}. In 1995, Godsil  proved that the number of vertices missed by a maximum matching in a graph \(G\) is the multiplicity of zero as a root of its matching polynomial, see \cite{godsil1995algebraic}. This result are all direct consequences of the Sachs theorem, see Theorem  3.8, pag.31 in \cite{bapat2014graphs}. Determine the family of \(n\)-trees that maximize \(m(T)\), the number of maximum matching in a tree \(T\), is a hard problem solved in see \cite{heuberger2011number}. The third corollary is a new way to think about this interesting problem: The number of maximum matching only depends on the S-set.

\begin{corollary}
	Let \(T\) be a tree. Then
	\[
	m(T)=\prod_{S \in\mathcal{F}_{S}(T) }m(S)
	\]
\end{corollary}

The tree \(T\) in Figure \ref{fig_el} has \(m(T)=m(S_{1})\,m(S_2)\,m(S_{3})=2\cdot3\cdot3=18\) maximum matchings. 

Use this corollary in order to characterizing the trees that maximize \(m(T)\) is a pending work.

As we pointed before, our work can be seen as a further step of the work of Nylen (\cite{nylen1998null}), and part of work of Neumaier (specifically, section 3 of \cite{neumaier1982second}); even though we were not aware of this former paper after finishing this work. The null decomposition of trees allow us to give counterexamples to Theorem 3.4 (ii) and Proposition 3.6 (ii)-(v) in \cite{neumaier1982second}. In order to do that we will write Neumaier result in our language.

In section 3 of \cite{neumaier1982second} were introduced the following notions associated to the possible zero entries of tree eigenvector. Let \(T\) be a tree, a vertex \(u \in V(T) \) is \(\lambda\)-essential if there is a \(\lambda\)-eigenvector \(x\) with \(x_{u}\neq 0\), thus \(0\)-essential  vertices are our supported vertices. The vertex \(u\) is \(\lambda\)-special if it is not essential, but neighbor of some essential vertex, and the vertex is \(\lambda\)-inessential otherwise. Thus \(0\)-special vertices are our core vertices and the \(0\)-inessential are N-vertices (vertices of the N-forest of \(T\)).

The part (ii) of Theorem 3.4 in \cite{neumaier1982second} says: Let \(T\) be a tree with \(\lambda\) an eigenvalue of \(T\) of multiplicity \(k\). If \(u\) is an inessential vertex of \(T\), then \(\lambda\) is an  eigenvalue of \(T-u\) of multiplicity \(k\). The tree in Figure ~\ref{Fig1} is a counterexample for this statement: the 0-essential vertices of \(T\) are \(\{2,3,4\}\), \(T\) has only one 0-special vertex: \(1\), and \(\{5,6\}\) are 0-inessential. The nullity of \(T\) is 2 and  the nullity of \(T-\{6\}\) is 3.

The tree in Figure \ref{Fig1} is also a counterexample to Proposition 3.6 parts (ii)-(v) in \cite{neumaier1982second}: Let \(T\) be a tree of order \(n\), then 
\begin{enumerate}
	\item[(ii)] No vertex of \(T\) is 0-inessential.
	\item[(iii)] A vertex is 0-special if and only if it is common to all maximum matching.
	\item[(iv)] An edge of \(T\) contains one or two 0-special vertices.
	\item[(v)] There are exactly \(\nu(T)\) 0-special vertices, and every edge of a maximum matching contains a unique 0-special vertex.
\end{enumerate}
All this statements are true if and only if \(T\) is an S-tree.